\documentclass[a4paper]{amsart}
\usepackage[T1]{fontenc}
\usepackage[latin1]{inputenc}
\usepackage{amsmath}
\usepackage{amssymb}
\usepackage{amsthm}
\usepackage[colorlinks,urlcolor=blue]{hyperref}

\newcommand{\N}{\mathbb{N}}

\newcommand{\CC}{\mathcal{C}}
\newcommand{\F}{\mathcal{F}}
\newcommand{\labeled}{\mathcal{L}}
\def\P{\mathcal{P}}

\newcommand{\Q}{\mathcal{P}}

\newcommand{\graphs}{\mathcal{U}}
\newcommand{\unlabeled}{\mathcal{U}}

\newcommand{\diff}{\mathrm{d}}

\newcommand{\abs}[1]{\lvert #1 \rvert}
\newcommand{\limitsx}[1]{\widehat{#1}}
\newcommand{\maxlimits}[1]{\widehat{#1}^*}

\DeclareMathOperator{\col}{col}
\DeclareMathOperator{\Ent}{Ent}
\DeclareMathOperator{\Forb}{Forb}

\newcommand\set[1]{\ensuremath{\{#1\}}}
\newcommand\oi{[0,1]}
\newcommand\ooi{[0,1)}

\newtheorem{theorem}{Theorem}

\newtheorem{lemma}[theorem]{Lemma}

\theoremstyle{remark}
\newtheorem{remark}[theorem]{Remark}

\title{On the Typical Structure of Graphs in a Monotone Property}

\author{Svante Janson \and Andrew J. Uzzell}
\address{Department of Mathematics, Uppsala University, P.O.~Box 480, SE-751 06 Uppsala, Sweden}
\email{\href{mailto:svante.janson@math.uu.se}{svante.janson@math.uu.se}}
\email{\href{mailto:andrew.uzzell@math.uu.se}{andrew.uzzell@math.uu.se}}

\date{April 9, 2014}

\thanks{Partly supported by the Knut and Alice Wallenberg Foundation.}

\begin{document}

\begin{abstract}
Given a graph property $\P$, it is interesting to determine the typical structure of graphs that satisfy $\P$.  In this paper, we consider monotone properties, that is, properties that are closed under taking subgraphs.  Using results from the theory of graph limits, we show that if $\P$ is a monotone property and $r$ is the largest integer for which every $r$-colorable graph satisfies $\P$, then almost every graph with~$\P$ is close to being a balanced $r$-partite graph.
\end{abstract}

\maketitle

\section{Introduction and main results}\label{se:intro}

Given a graph property~$\P$, it is natural to study the structure of a typical graph that satisfies $\P$.  A graph property is \emph{monotone} if it is closed under taking subgraphs and \emph{hereditary} if it is closed under taking induced subgraphs.  Thus, every monotone property is also hereditary.  Many authors have studied the structure of typical graphs in various hereditary properties---see, e.g.,~\cite{ABBM11,BBS09,EFR86,EKR76,HJS,JUstring}, as well as the survey~\cite{Bol07}. In this note, we use results from graph limit theory to study the structure of a typical graph in a general monotone property.

Before stating our main result, let us recall certain basic notions of graph limit theory.  For more details, see, e.g., \cite{BCLSV1,DJ08,LS06}, as well as the monograph~\cite{LovaszBook}.  Here, we simply recall that certain sequences of graphs are defined to be \emph{convergent}.  A convergent sequence has a limit, called a \emph{graph limit}, which is unique if it exists.

Lov\'asz and Szegedy~\cite{LS06} showed that a graph limit~$\Gamma$ may be represented by a \emph{graphon}, a symmetric, measurable function~$W : \oi^2 \to \oi$.  (So, abusing notation slightly, we will sometimes write $G_n \to W$ if the sequence $\set{G_n}_{n = 1}^{\infty}$ converges to the graph limit~$\Gamma$ represented by $W$.)  More than one graphon may represent the same graph limit; we say that the graphons $W_1$ and~$W_2$ are \emph{equivalent}, and write $W_1 \cong W_2$, if they represent the same graph limit.

Let $X_1$, $X_2, \dots$ be i.i.d.~uniformly distributed random variables in $\oi$.  Given a graphon $W$, the \emph{$W$-random graph} $G(n, W)$ is a graph with vertex set~$[n]$ in which vertices $i$ and~$j$ are adjacent with probability~$W(X_i, X_j)$, independently of all other edges.

Let $h(x) = -x \log_2(x) - (1 - x)\log_2(1 - x)$ denote the binary entropy function.  The \emph{entropy} of a graphon~$W$ is
\[
\Ent(W) = \int_0^1 \int_0^1 h\bigl(W(x,y)\bigr) \, \diff\mu(x)\diff\mu(y),
\]
where $\mu$ denotes the Lebesgue measure.  As noted in~\cite{HJS}, if $W_1 \cong W_2$, then $\Ent(W_1) = \Ent(W_2)$.  In other words, entropy is a property of a graph limit, rather than of the graphon that represents it.  Thus, we may define the entropy~$\Ent(\Gamma)$ of a graph limit~$\Gamma$ to be the entropy of any graphon that represents it.

Hatami, Janson, and Szegedy~\cite{HJS} posed the question of which graphons may arise as limits of sequences of graphs with a given property~$\P$.  In addition to the intrinsic interest of this question, it turns out that if $\P$ is hereditary, then certain limits of sequences of graphs in $\P$ (namely, those with maximum entropy) give a great deal of information about the number and typical structure of graphs in~$\P$. (We do not distinguish between a graph property and the class of graphs with that property.)  In order to state these results, we need to introduce more notation.

Let $\graphs_n$ denote the set of unlabeled graphs on $n$ vertices and let $\labeled_n$ denote the set of labeled graphs with vertex set~$[n]$.  Given a graph property~$\P$, we let $\P_n = \P \cap \graphs_n$ denote the set of unlabeled elements of~$\P$ with $n$ vertices and let $\P^L_n$ denote the set of labeled elements of~$\P$ with vertex set~$[n]$.  The function $n \mapsto \abs{\P_n}$ is called the (unlabeled) \emph{speed} of~$\P$; the labeled speed is defined similarly.  Observe that
\begin{equation}\label{eq:labeledunlabeled}
\bigl\lvert \P_n \bigr\rvert \leq \bigl\lvert \P^L_n \bigr\rvert \leq n! \bigl\lvert \P_n \bigr\rvert.
\end{equation}

Given a graph property~$\P$, we let $\limitsx{\P}$ denote the set of graph limits of sequences in $\P$.  We furthermore let $\maxlimits{\P}$ denote the set of elements of $\limitsx{\P}$ of maximum entropy, i.e.,
\[
\maxlimits{\P} = \biggl\{\Gamma \in \limitsx{\P} \, : \, \Ent(\Gamma) = \max_{\Gamma' \in \limitsx{\P}} \Ent(\Gamma')\biggr\}.
\]
We will also use these symbols to refer to the set of graphons (respectively, the set of maximum-entropy graphons) representing limits of sequences in $\P$.  It is shown in~\cite{HJS} that if $\Q$ is hereditary (and not finite), then $\max_{\Gamma \in \limitsx{\Q}} \Ent(\Gamma)$ is achieved---in other words, $\maxlimits{\Q}$ is nonempty.

In~\cite[Theorem 1.6]{HJS}, Hatami, Janson, and Szegedy showed that if a hereditary property~$\P$ has a single graph limit~$\Gamma$ of maximum entropy, then a typical element of~$\P$ is close to~$\Gamma$ (in terms of the standard cut metric on the space of graph limits).

\begin{theorem}\label{th:problimit}
Suppose that $\Q$ is a hereditary property and that $\max_{\Gamma \in \limitsx{\Q}} \Ent(\Gamma)$ is attained by a unique graph limit~$\Gamma_{\Q}$.  Then
\begin{enumerate}
\item[(i)] if $G_n \in \unlabeled_n$ is a uniformly random unlabeled element of~$\Q_n$, then $G_n$ converges in probability to~$\Gamma_{\Q}$ as $n \to \infty$;
\item[(ii)] if $G_n \in \labeled_n$ is a uniformly random labeled element of~$\Q^L_n$, then $G_n$ converges in probability to~$\Gamma_{\Q}$ as $n \to \infty$.
\end{enumerate}
\end{theorem}

Now we define a special class of graphons.  All of these graphons will be defined on $\ooi^2$, rather than on $\oi^2$; it is easy to see that this change is immaterial. Given $r \in \N$ and $i \in [r]$, let $I_i = [(i - 1)/r, i/r)$ and let $E_r = \cup_{i \neq j} I_i \times I_j$.  We also let $E_{\infty} = \ooi^2$.  Given $r \in \N \cup \set{\infty}$, we let $R_r$ denote the set of graphons $W$ such that $W(x, y) = 1/2$ if $(x, y) \in E_r$ and $W(x, y) \in \set{0, 1}$ otherwise.  It is easy to see that if $W \in R_r$, then
\[
\Ent(W) = \iint_{E_r} h(1/2)\, \diff\mu(x)\diff\mu(y) = \mu(E_r) = 1  - \dfrac{1}{r}.
\]
For $r \in \N$ and $0 \leq s \leq r$, we let $W^*_{r,s}$ denote the graphon in $R_r$ that equals~$1$ on $I_i \times I_i$ for $i \leq s$ and equals~$0$ on $I_i \times I_i$ for $s + 1 \leq i \leq r$.  Observe that $R_{\infty}$ consists only of the graphon that equals~$1/2$ everywhere on $\ooi^2$; for notational convenience, we denote this graphon by $W^*_{\infty, 0}$.

Given $r \in \N$ and $0 \leq s \leq r$, we let $\CC(r, s)$ denote the class of graphs whose vertex sets can be partitioned into $s$ (possibly empty) cliques and $r - s$ (possibly empty) independent sets.  In particular, $\CC(r, 0)$ is the class of $r$-colorable graphs.  Observe that for each $r$ and~$s$, the class~$\CC(r,s)$ is hereditary, and that $\CC(r, 0)$ is monotone.

It is shown in~\cite[Theorem 1.9]{HJS} that if $\P$ is a hereditary property, then the maximum entropy of an element of~$\limitsx{\P}$ takes one of countably many values, and furthermore that this value determines the asymptotic speed of~$\P_n$.

\begin{theorem}\label{th:speedlimit}
If $\P$ is a hereditary property, then there exists $r \in \N \cup \set{\infty}$ such that  $\max_{\Gamma \in \limitsx{\P}} \Ent(\Gamma) = 1 - 1/r$ and such that every graph limit $\Gamma \in \maxlimits{\P}$ can be represented by a graphon~$W \in R_r$.  Moreover,
\begin{equation*}\label{eq:speedlimit}
\abs{\P_n} = 2^{\bigl(1 - \frac{1}{r} + o(1) \bigr) \binom{n}{2}}.
\end{equation*}
\end{theorem}

Given a graph~$F$, we say that a graph~$G$ is \emph{$F$-free} if no subgraph of~$G$ is isomorphic to~$F$.  Given a (possibly infinite) family of graphs~$\F$, we say that $G$ is \emph{$\F$-free} if it is $F$-free for every $F \in \F$.  Observe that for any family~$\F$, the class of $\F$-free graphs is monotone.  (Conversely, every monotone class $\P$ equals the class of $\F$-free graphs for some family~$\F$---for example, $\F = \graphs \setminus \P$.)  We write $\Forb(\F)$ for the class of $\F$-free graphs and write $\Forb(F)$ when $\F = \set{F}$.  Note in particular that $\Forb(\emptyset)$ equals the class of all unlabeled finite graphs, which we denote by $\graphs$.

The \emph{coloring number} of a family of graphs~$\F$ is
\[
	\col(\F) = \min_{F \in \F} \chi(F).
\]
In particular, we define
\begin{equation}\label{eq:emptycolnumber}
\col(\emptyset) = \infty.
\end{equation}
Our main result says that if $\col(\F) = r + 1$, then a typical element of~$\Forb(\F)$ resembles a balanced $r$-partite graph in which cross-edges are present independently with probability~$1/2$.

\begin{theorem}\label{th:Ffreelimit}
Let $\F$ be a family of graphs and let $r = \col(\F) - 1$.  If $\P = \Forb(\F)$, then as $n$ tends to~$\infty$, a sequence of uniformly random unlabeled (respectively, labeled) elements of~$\P_n$ (respectively, elements of~$\P^L_n$) converges in probability to the graph limit~$\Gamma_r$ represented by $W^*_{r, 0}$.
\end{theorem}

Note that the quantity~$r$ in the statement of the theorem also equals the largest integer~$t$ for which every $t$-colorable graph is in $\Forb(\F)$.

It follows from Theorems \ref{th:speedlimit} and~\ref{th:Ffreelimit} that if $\col(\F) = r + 1$ then
\begin{equation}\label{eq:ForbFspeed}
\bigl\lvert \Forb(\F)_n \bigr\rvert = 2^{\bigl(1 - \frac{1}{r} + o(1) \bigr) \binom{n}{2}}.
\end{equation}
Let us note that Balogh, Bollob\'as, and Simonovits~\cite{BBS04} obtained a fairly sharp bound on the error term in~\eqref{eq:ForbFspeed}.

\begin{remark}
The proof of Theorem~\ref{th:Ffreelimit} shows that if $r = \col(\F) - 1$ and $\P = \Forb(\F)$, then $W^*_{r,0}$ is the unique maximum-entropy element of~$\limitsx{\P}$.  For certain families~$\F$, it is also possible to describe the set of all $\F$-free graph limits.  For example, the set of limits of bipartite graphs is determined in~\cite[Example 2.1]{HJS}, and a very similar argument holds for $r$-partite graphs when $r \geq 3$.  However, we know of no representation of the set of all $\F$-free graph limits for arbitrary $\F$.
\end{remark}

\begin{remark}
Erd\H{o}s, Frankl, and R\"{o}dl~\cite{EFR86} showed that if $\chi(F) = r + 1$, then every $F$-free graph $G$ may be made $K_{r+1}$-free by removing $o(n^2)$ edges from $G$.  This result is similar in spirit to Theorem~\ref{th:Ffreelimit}, but we see no direct implication: if $\set{G_n}_{n = 1}^{\infty}$ is a sequence of uniformly random $F$-free graphs and $\set{G'_n}_{n = 1}^{\infty}$ is the sequence of resulting $K_{r+1}$-free graphs, then the distribution of $G'_n$ need not be uniform in $\Forb(K_{r+1})_n$.
\end{remark}

\begin{remark}
Theorem~\ref{th:Ffreelimit} says that if $\col(\F) = r + 1$ then almost every (labeled or unlabeled) $\F$-free graph is close to a balanced $r$-partite graph.  (Conversely, every $r$-partite graph is trivially $\F$-free.) In the case of labeled graphs, Pr\"{o}mel and Steger~\cite{PS92} proved a much stronger result for a specific class of monotone properties: they characterized the graphs~$F$ for which almost every labeled $F$-free graph \emph{is} $(\chi(F) - 1)$-partite.  Given a graph~$F$, we say that $e \in E(F)$ is \emph{critical} if $\chi(F - e) < \chi(F)$.  Pr\"{o}mel and Steger showed that if $\chi(F) = r + 1$ then
\[
\bigl\lvert \Forb(F)^L_n \bigr\rvert = \bigl(1 + o(1)\bigr) \bigl\lvert \CC(r, 0)^L_n \bigr\rvert
\]
if and only if $F$ contains a critical edge.  They also showed that if $F$ does not contain a critical edge, then there exists a constant~$c_r > 0$ not depending on~$F$ such that
\begin{equation}\label{eq:nocriticaledge}
\bigl\lvert \Forb(F)^L_n \bigr\rvert \geq c_r n \bigl\lvert \CC(r, 0)^L_n \bigr\rvert
\end{equation}
for all $n$~sufficiently large.  Theorem~\ref{th:Ffreelimit} shows that if $\F$ is any family of graphs with $\col(\F) = r + 1$, then $\Forb(\F)^L$ and $\CC(r, 0)^L$ have roughly the same asymptotic speed.  Note that this result does not contradict~\eqref{eq:nocriticaledge} when $\F = \{F\}$ and $F$ does not contain a critical edge: if $\chi(F) = r + 1$, then, in view of \eqref{eq:labeledunlabeled} and  \eqref{eq:ForbFspeed}, Theorem~\ref{th:Ffreelimit} implies the weaker statement that $\abs{\Forb(F)^L_n}$ and $\abs{\CC(r, 0)^L_n}$ differ by a factor of~$2^{o(n^2)}$.
\end{remark}

\section{Proof of Theorem \ref{th:Ffreelimit}.}\label{se:proof}

\begin{lemma}\label{le:monotonicity}
Let $\Q$ be a monotone property and let $W \in \limitsx{\Q}$.  If $W'$ is a graphon such that $W' \leq W$ pointwise, then $W' \in \limitsx{\Q}$.
\end{lemma}

\begin{proof}
Consider the sequences of random graphs $\{G(n, W)\}_{n=1}^{\infty}$ and~$\{G(n, W')\}_{n=1}^{\infty}$.  Since $W' \leq W$ pointwise, a standard argument shows that the two sequences can be coupled so that for each~$n$, $G(n, W') \subseteq G(n, W)$ almost surely.  It is shown in~\cite[Theorem 3.1]{Jan13b} that if $W \in \limitsx{\Q}$ then, for each $n$, $G(n, W) \in \Q$ almost surely.  It follows that for each $n$, we almost surely have $G(n, W') \in \Q$, as well.  Finally, it is shown in~\cite[Theorem 4.5]{BCLSV1} that $G(n, W') \to W'$ almost surely as $n \to \infty$, which implies that $W' \in \limitsx{\Q}$, as claimed.
\end{proof}

Now we prove our main result, Theorem~\ref{th:Ffreelimit}.

\begin{proof}[Proof of Theorem~\ref{th:Ffreelimit}.]
We begin by showing that, up to equivalence of graphons, $\limitsx{\P}$ contains a unique element of maximum entropy.  By Theorem~\ref{th:speedlimit}, there exists~$t \in \N \cup \set{\infty}$ such that $\maxlimits{\P} \subseteq R_t$ up to equivalence of graphons.

First, suppose that $t < \infty$.  Observe that if $W \in \maxlimits{\P} \cap R_t$, then $W \geq W^*_{t,0}$ pointwise, which by Lemma~\ref{le:monotonicity} implies that $W^*_{t,0} \in \maxlimits{\P}$.  We claim that, up to equivalence of graphons, $W^*_{t,0}$ is in fact the only maximum-entropy element of~$\limitsx{\P}$.  Indeed, let $W' \in \maxlimits{\P} \cap R_t$ and suppose that $\mu(W' = 1) > 0$.  But then Lemma~\ref{le:monotonicity} implies that $\min\{W', 1/2\}$ is a graphon in $\limitsx{\P}$ with entropy strictly larger than~$1 - 1/t$, which contradicts the definition of~$t$.  Thus, $W' = W^*_{t, 0}$ a.e.; in particular, $W' \cong W^*_{t, 0}$.  If $t = \infty$, then (up to equivalence) $\maxlimits{\P}$ must consist of the sole element of~$R_{\infty}$, that is, the graphon~$W^*_{\infty, 0}$.

Now we show that $t = \col(\F) - 1 = r$.  Suppose that $t < \infty$.  It is observed in~\cite[Remark 1.10]{HJS} that if $\P$ is hereditary and $0 \leq s \leq r < \infty$, then $W^*_{r,s} \in \limitsx{\P}$ if and only if $\CC(r, s) \subseteq \P$.  By the definition of~$\col(\F)$, it is easy to see that if $u \leq r$, then $\CC(u, 0) \subseteq \P$ and hence $W^*_{u,0} \in \limitsx{\P}$.  On the other hand, $\F$ contains some element of~$\CC(r + 1, 0)$, which implies that $\CC(r + 1, 0) \nsubseteq \P$.  This implies that $W^*_{u,0} \notin \limitsx{\P}$ when $u \geq r + 1$, and hence that $t = r$.

If $t = \infty$, then we claim that $\P = \Forb(\emptyset) = \graphs$; the conclusion then follows from~\eqref{eq:emptycolnumber}.  Suppose to the contrary that $\P$ does not contain some graph~$F$.  Then $\CC(\chi(F), 0) \nsubseteq \P$, which implies that $W^*_{\chi(F), 0} \notin \limitsx{\P}$.  However, $W^*_{\chi(F), 0} \leq W^*_{\infty, 0}$ pointwise, so Lemma~\ref{le:monotonicity} implies that $W^*_{\infty, 0} \notin \limitsx{\P}$, which is a contradiction.

Finally, since $\P$ is a hereditary property, it follows from Theorem~\ref{th:problimit} that a uniformly random (labeled or unlabeled) element of~$\P$ converges in probability to~$\Gamma_r$, as claimed.
\end{proof}

\bibliographystyle{amsplain}
\bibliography{graphlimitsbib}

\end{document}